\documentclass{amsart} \usepackage{hyperref}
\usepackage{enumerate} 
\usepackage{upref}
\usepackage{verbatim} 
\usepackage{color}
\usepackage{graphicx}
\usepackage{todonotes}

\newcommand*{\mailto}[1]{\href{mailto:#1}{\nolinkurl{#1}}}

\usepackage[latin1]{inputenc}
\usepackage{amssymb, amsmath, amsfonts, amsthm}
\usepackage[all]{xy}

\usepackage{marginnote}

\newcommand{\R}{\mathbb R}

\newcommand{\norm}[1]{\left\Vert#1\right\Vert}

\newcommand{\nn}{\nonumber}

\newtheorem{theorem}{Theorem}[section]
\newtheorem{corollary}[theorem]{Corollary}

\newtheorem{lemma}[theorem]{Lemma}
\newtheorem{definition}[theorem]{Definition}

\numberwithin{equation}{section}

\def\bel{\begin{equation}\label}
\def\eeq{\end{equation}}
\def\ds{\displaystyle}
\def\ov{\overline}
\def\bega{\begin{array}}
\def\enda{\end{array}}
\def\ve{\varepsilon}

\allowdisplaybreaks

\begin{document}

\title[Burgers-Poisson equation]{On the Burgers-Poisson equation}

\author[K. Grunert]{Katrin Grunert}
\address{Department of Mathematical Sciences\\ Norwegian University of Science and Technology\\ NO-7491 Trondheim\\ Norway}
\email{\mailto{katring@math.ntnu.no}}
\urladdr{\url{http://www.math.ntnu.no/~katring/}}

\author[K. T. Nguyen]{Khai T. Nguyen}
\address{Department of Mathematics\\ Penn State University\\ PA 16802\\ USA}
\email{\mailto{ktn2@psu.edu}}
\urladdr{\url{http://www.personal.psu.edu/ktn2/}}

\subjclass[2010]{Primary: 35L03; Secondary: 35L67}
\keywords{Burgers-Poisson equation, existence, uniqueness, blow-up}
\thanks{Research supported by the grant {\it Waves and Nonlinear Phenomena (WaNP)} from the Research Council of Norway}

\begin{abstract}
In this paper, we prove  the existence and uniqueness of weak entropy solutions to the Burgers-Poisson equation for initial data in ${\bf L}^1(\mathbb{R})$. In addition an Oleinik type estimate is established and some criteria on local smoothness and wave breaking for weak entropy solutions are provided.
\end{abstract}

\maketitle

\section{Introduction}
\setcounter{equation}{0}
The Burgers--Poisson equation is given by the balance law obtained from Burgers' equation by adding
a nonlocal source term 
\bel{BP} 
u_t + \left({\frac{u^2}{ 2}}\right)_x=[G*u]_x\,,
\eeq
where
$$G(x)=-{\frac12}e^{-|x|}\quad\text{and}\quad [G*u](x)=-{\frac12}\int^{\infty}_{-\infty}e^{-|x-z|}\cdot u(z)dz$$
solve the Poisson equation
\[
\varphi_{xx}-\varphi~=u\,.
\]
Equation (\ref{BP}) has been derived in \cite{W} as a simplified model of shallow water waves and admits conservation of both momentum and energy. For sufficiently smooth initial data
\bel{id}u(0,x)~=\ov u(x)\,,\eeq
the local existence and uniqueness of solutions of (\ref{BP}) have been established in \cite{FC}. In addition the analysis of traveling waves showed that the equation features
wave breaking in finite time. Hence it is natural to study existence and uniqueness of weak entropy solutions which are global in time.
\begin{definition}
A function $u\in \mathbf{L}_{loc}^{1}([0,\infty[\times\R)\cap {\bf L}^{\infty}_{loc}(]0,\infty[, {\bf L}^{\infty}(\R))$ is a {\bf weak entropy solution} of \eqref{BP}-\eqref{id} if $u$ satisfies the following properties:
\quad\\
\begin{itemize}
\item [(i)] The map $t\mapsto u(t,\cdot)$ is continuous with values in $\mathbf{L}^1(\R)$ and satisfies the initial condition \eqref{id}.
\quad\\
\item[(ii)] For any $k\in\R$ and any non-negative test function $\phi\in C^1_c(]0,\infty[\times\R,\R)$ one has
\bel{entropy}
\ds\int\int \Big[|u-k|\phi_t+\mathrm{sign}(u-k)\Big({u^2\over 2}-{k^2\over 2}\Big)\phi_x+\mathrm{sign}(u-k)[G_x*u(t,\cdot)](x)\phi\Big]~dxdt~\geq~0\,.
\eeq
\end{itemize}
\end{definition}
\quad\\
\noindent
For any initial data $\bar{u}\in BV(\R)$, the existence of a global weak entropy solution to \eqref{BP}-\eqref{id} has been studied in \cite{FC}. The proof is based on the vanishing viscosity method yielding a sequence of approximating smooth solutions. Due to the BV bound of $\bar u$, one obtains that the approximating solutions also satisfy an a priori uniform BV bound for all positive times, yielding the compactness of the approximating sequence of solutions. However this method cannot be applied in the more general case with initial data in ${\bf L}^1(\R)$. In addition there are no uniqueness or continuity results for global weak entropy solutions of \eqref{BP} established in \cite{FC}.
Thus our main goal is to study the existence and uniqueness for global weak entropy solutions for initial data $\bar u\in {\bf L}^1(\R)$. To be more explicit we are going to show the following theorem.
\begin{theorem}\label{thm1}
Given any initial data $u(0,\cdot)=\bar{u}(\cdot)\in \mathbf{L}^1(\R)$, the Cauchy problem \eqref{BP}-\eqref{id} has a unique weak entropy solution $u(t,x)$ in $[0,\infty)\times\R$. Furthermore, for any $t>0$ 
\bel{L-1}
\norm{u(t,\cdot)}_{\mathbf{L}^1(\R)}\leq e^{t} \norm{\bar{u}}_{\mathbf{L}^1(\R)}\,,
\eeq
and
\[
u(t,y)-u(t,x)\leq\Big[\frac1{t}+2+2 t+4te^{t} \norm{\bar{u}}_{\mathbf{L}^1(\R)}\Big] (y-x),\qquad x<y\,. 
\]
Moreover, let $v(t,x)$ be the weak entropy solution of \eqref{BP} with initial data $v(0,\cdot)=\bar{v}(\cdot)\in\mathbf{L}^1(\R)$. Then, for every $t>0$, it holds
\[
\norm{u(t,\cdot)-v(t,\cdot)}_{\mathbf{L}^1(\R)}\leq e^t\norm{\bar{u}-\bar{v}}_{\mathbf{L}^1(\R)}\,.
\]
\end{theorem}
\quad\\
\noindent The above solutions will be constructed by a flux-splitting method. Relying on the decay properties of the semigroup generated by Burgers' equation \cite{M,O} and the Lipschitz continuity of solutions to the  Poisson equation (see Lemma~\ref{lem2}), we prove that the approximating solutions satisfy an Oleinik type inequality. As a consequence the sequence of approximating solutions is precompact and converges in ${\bf L}^1_{\mathrm{loc}}(\R)$. Moreover, we show using an energy estimate, that the characteristics are H{\"o}lder continuous. This allows us to derive a {\it Tightness Property} for the sequence of approximating solutions, which implies the continuity property of the solutions. 

It is well-known that the nonlocal Poisson source term in the Burgers--Poisson equation cannot prevent the breaking induced from the Burgers term $uu_x$. This means it is possible that the velocity slope $u_x$ blows up in finite time even if the initial data is very smooth and has small $C^1$-norm. In \cite{H-Liu}, a criterion on wave breaking has been established  in the class of spatially periodic solutions. By a careful study of the derivative of the solution $u_x$ along characteristics, 
we extend the result in  \cite{H-Liu} to the general case (see Theorem~\ref{WB}). Furthermore, we provide  some criteria on local smoothness. In particular, we prove that if the ${\bf L}^{\infty}$-norm of the derivative $u_{0,x}$ of the initial data $u_0$ is small then the corresponding weak entropy solution of (\ref{BP}) will remain smooth for a large time. 
\quad\\
\\
The paper is organized as follows. In Section \ref{sec:2} we construct approximating solutions and provide some a priori estimates.  In Section \ref{sec:3} we will prove the existence and uniqueness for weak entropy solutions. Finally, we are going to study local smoothness and wave breaking criteria for weak entropy solutions.

\section{Approximating solutions and some a priori estimates}
\label{sec:2}

The proof of the existence and uniqueness of weak entropy solutions to the Burgers-Poisson equation with initial data in $\mathbf{L}^1(\R)$ is based on a limiting process for approximating solutions, which are constructed by the flux-splitting method. Thus this section is concerned with the construction of the approximating solutions as well as the derivation of some a priori estimates for them.

\setcounter{equation}{0}

\quad\\
{\bf 1. Approximating solutions.} 
For some fixed integer $\nu\geq 1$, we define the time steps
\[
t_i=i\cdot 2^{-\nu},\qquad i=0,1,2,\dots.
\]
The approximating solution of (\ref{BP}) is then defined inductively as
\bel{unudef}\left\{\bega{l} u^\nu(0)=\bar u,\qquad\qquad u^\nu(t_{i}) = u^\nu(t_{i} -)+2^{-\nu}\cdot
[G_x*u^\nu(t_i-)],\qquad i=1,2,\ldots,\cr\cr
u^\nu(t)=S^B_{t-t_i}(u^\nu(t_i)),\qquad\qquad 
t\in [t_i, \, t_{i+1}[\,,\quad i=0,1,2,\ldots.\enda\right.\eeq
Here $S^B$ denotes the semigroup generated by Burgers' equation. 
More precisely,
$t\mapsto S^B_t(\bar u)$ denotes the Kruzkov entropy solution to
\bel{CPB}
u_t+\left(\frac{u^2}{2}\right)_x=0\qquad\qquad u(0,x)=\bar u(x)\in 
\mathbf{L}^1(\R)\,.\eeq

\noindent For every $u\in \mathbf{L}^1(\R)$, we have
\bel{L-1B}
\norm{S^B_t(u)}_{\mathbf{L}^1(\R)}\leq\norm{u}_{\mathbf{L}^1(\R)},\qquad\text{for all } t\geq0\,,
\eeq
and 
\bel{L-1G}
\norm{G_x*u}_{\mathbf{L}^1(\R)}\leq\norm{u}_{\mathbf{L}^1(\R)}\,,
\eeq
which implies
\begin{eqnarray*}
\norm{u^{\nu}(t_{i})}_{\mathbf{L}^{1}(\R)}&\leq&\norm{S^B_{2^{-\nu}}(u^{\nu}(t_{i-1}))}_{\mathbf{L}^{1}(\R)}+2^{-\nu} \norm{[G_x*S^B_{2^{-\nu}}(u^{\nu}(t_{i-1}))]}_{\mathbf{L}^{1}(\R)}\\
\\
&\leq&(1+2^{-\nu})\norm{u^{\nu}(t_{i-1})}_{\mathbf{L}^{1}(\R)}\leq(1+2^{-\nu})^{i}\norm{\bar{u}}_{\mathbf{L}^{1}(\R)}\\
\\
&=&(1+2^{-\nu})^{2^{\nu} t_i}\norm{\bar{u}}_{\mathbf{L}^{1}(\R)}\leq e^{t_i} \norm{\bar{u}}_{\mathbf{L}^{1}(\R)}\,.
\end{eqnarray*}
By \eqref{L-1B}, we obtain that 
\bel{L-1-A}
\norm{u^{\nu}(t)}_{\mathbf{L}^1(\R)}\leq e^{t}\norm{\bar{u}}_{\mathbf{L}^{1}(\R)},\qquad \text{for all }t\geq 0\,.
\eeq
\quad\\

\noindent {\bf 2. Oleinik type inequality.} We claim that for any $i\geq 1$ and $t\in [t_i,t_{i+1}[$ it holds that 
\bel{Pos-D}
u^{\nu}(t,x_2)-u^{\nu}(t,x_1)\leq\Big[\frac1{t_i}+2+2 t_i+4e^{t_i}t_i\norm{\bar{u}}_{\mathbf{L}^1(\R)}\Big] \cdot (x_2-x_1),\quad\text{for all } x_1<x_2\,.
\eeq
The proof relies heavily on the positive decay of Burgers' semigroup and the Lipschitz continuity of solutions to the Poisson equation.

\begin{lemma}\label{lem1}
Let $u_0\in \mathbf{L}^1(\R)$ be such that
\bel{as1}
u_0(x_2)-u_0(x_1)\leq K \cdot (x_2-x_1),\qquad\qquad\text{for all } x_1<x_2\,,
\eeq
for some constant $K>0$.
Then
\bel{O-L1}
S^B_t(u_0)(x_2)-S^B_t(u_0)(x_1)\leq\frac{K}{1+Kt}\cdot (x_2-x_1),\qquad\text{for all } x_1<x_2\,.
\eeq
 \end{lemma}

 \begin{proof}
 It is sufficient to prove \eqref{O-L1} for any point of continuity $x_i$  of $S^B_t(u_0)$. Let $\xi_{x_i}(\cdot)$ be the characteristic through the point $(t,x_i)$, then we have 
 \[
 x_i=\xi_{x_i}(0)+t u_0(\xi_{x_i}(0))\quad\text{and}\quad  S^B_t(u_0)(x_i)~=~u_0(\xi_{x_i}(0))\,.
 \]
From the assumption \eqref{as1}, we get
\begin{eqnarray*}
 x_2-x_1&=&\xi_{x_2}(0)-\xi_{x_1}(0)+t\cdot (u_0(\xi_{x_2}(0))-u_0(\xi_{x_1}(0)))\\
 \\
 &\leq& (1+Kt)\cdot \big[\xi_{x_2}(0)-\xi_{x_1}(0)\big]\,,
 \end{eqnarray*}
which implies that 
\begin{eqnarray*}
S^B_t(u_0)(x_2)-S^B_t(u_0)(x_1)&=&\frac1{t}\cdot \big[(x_2-x_1)-(\xi_{x_2}(0)-\xi_{x_1}(0))\big]\\
\\
&\leq& \frac{K}{1+Kt}\cdot (x_2-x_1)\,,
\end{eqnarray*}
and the proof is complete.
\end{proof}

\begin{lemma}\label{lem2}
Let $u_0\in \mathbf{L}^1(\R)$ be such that
\bel{as2}
u_0(x_2)-u_0(x_1)\leq K\cdot (x_2-x_1),\qquad\qquad\text{for all }x_1<x_2\,,
\eeq
for some constant $K>0$.
Then
\bel{L-G}
\big|[G_x*u_0](x_2)-[G_x*u_0](x_1)\big|\leq~\Big[\norm{u_0}_{\mathbf{L}^1(\R)}+ \sqrt{2K\norm{u_0}_{\mathbf{L}^1(\R)}}~\Big] \cdot |x_2-x_1|\,.
\eeq
\end{lemma}

\begin{proof}
For any $x_1<x_2$, we compute 
\begin{multline*}
\big|[G_x*u_0](x_2)-[G_x*u_0](x_1)\big|~\leq~ \frac12 \int^{x_1}_{-\infty}|e^{z-x_1}-e^{z-x_2}|\cdot |u_0(z)|~dz\\+\frac12\int_{x_2}^{\infty}|e^{x_1-z}-e^{x_2-z}|\cdot |u_0(z)|~dz+\frac12\int_{x_1}^{x_2}|e^{x_1-z}+e^{z-x_2}|\cdot |u_0(z)|~dz\\~\leq~\Big[\norm{u_0}_{\mathbf{L}^1}+\norm{u_0}_{\mathbf{L}^\infty}\Big]\cdot |x_2-x_1|\,.
\end{multline*}
Concluding as in the proof of \cite[Lemma 4.2]{FOK}, \eqref{as2} implies that 
\bel{L-I-A}
\norm{u_0}_{\mathbf{L}^\infty}\leq\sqrt{2K\norm{u_0}_{\mathbf{L}^1(\R)}}~\,,
\eeq
and hence (\ref{L-G}).
\end{proof}

\noindent
Using Lemma~\ref{lem1} and Lemma~\ref{lem2}, we now show by induction that for any $i=1,2, \dots$, one has
\bel{OL-1}
u^{\nu}(t_i-,x_2)-u^{\nu}(t_i-,x_1)\leq a_i\cdot(x_2-x_1),\qquad\text{for all } x_1<x_2\,,
\eeq
where 
\[
a_{1}=2^{\nu}\quad\text{and}\quad a_{i+1}=\frac{(1+2^{-\nu})\cdot a_i+2^{-\nu+1}e^{t_i} \norm{\bar{u}}_{\mathbf{L}^1(\R)}}{1+ \Big[(1+2^{-\nu})\cdot a_i+2^{-\nu+1} e^{t_i} \norm{\bar{u}}_{\mathbf{L}^1(\R)} \Big]\cdot 2^{-\nu}}\,.
\]
Indeed, since $u^{\nu}(t_1-,\cdot)=S^B_{2^{-\nu}}(\bar{u})(\cdot)$, \eqref{OL-1} holds for $i=1$ by Oleinik's inequality, see e.g. \cite[Chapter 3.4]{E}.  Assume that \eqref{OL-1} holds for all indices up to $i$. 
Then it follows from Lemma~\ref{lem2} and \eqref{L-1-A}, that
\begin{multline}\label{onesided}
u^{\nu}(t_i,x_2)-u^{\nu}(t_i,x_1)~=~u^{\nu}(t_i-,x_2)-u^{\nu}(t_i-,x_1)
 \\
\qquad\qquad\qquad\qquad\qquad\qquad + 2^{-\nu}\cdot\big (
[G_x*u^\nu (t_i-)](x_2)-[G_x*u^\nu (t_i-)](x_1)\big) \cr
\qquad\qquad\leq~\Big(a_i+2^{-\nu} \Big(\norm{u^\nu(t_i-)}_{\mathbf{L}^1(\R)}+ \sqrt{2a_i\norm{u^\nu(t_i-)}_{\mathbf{L}^1(\R)}}~\Big)\Big)\cdot (x_2-x_1)\cr
 \qquad\qquad\leq~\Big(a_i+2^{-\nu}\Big(e^{t_i}\norm{\bar{u}}_{\mathbf{L}^1(\R)}+ \sqrt{2e^{t_i}a_i \norm{\bar{u}}_{\mathbf{L}^1(\R)}}\Big)\Big)\cdot (x_2-x_1)\cr
 ~\leq~\Big((1+2^{-\nu}) a_i+2^{-\nu+1} e^{t_i} \norm{\bar{u}}_{\mathbf{L}^1(\R)} \Big)\cdot (x_2-x_1) 
\end{multline}
for any $x_1<x_2$.\\
\quad\\
\noindent Applying Lemma~\ref{lem1} to $u_0(\cdot)=u^{\nu}(t_i,\cdot)$ and $t=2^{-\nu}$, we obtain 
\begin{multline*}
u^{\nu}(t_{i+1}-,x_2)-u^{\nu}(t_{i+1}-,x_1)~=~S^B_{2^{-\nu}}(u^{\nu}(t_i))(x_2)-S^B_{2^{-\nu}}(u^{\nu}(t_i))(x_1)\\
\\
~\leq~\frac{(1+2^{-\nu}) a_i+2^{-\nu+1} e^{t_i} \norm{\bar{u}}_{\mathbf{L}^1(\R)}}{1+ \Big[(1+2^{-\nu}) a_i+2^{-\nu+1} e^{t_i} \norm{\bar{u}}_{\mathbf{L}^1(\R)} \Big] \cdot 2^{-\nu}}\cdot(x_2-x_1)\,,
\end{multline*}
for any $x_1<x_2$, which is (\ref{OL-1}) for $i+1$.\\

\noindent
Note that \eqref{onesided} together with Lemma~\ref{lem1} implies, for all $t\in[t_i,t_{i+1}[$, that 
\begin{equation*}
u^{\nu}(t,x_2)-u^{\nu}(t,x_1)~\leq~ \big[(1+2^{-\nu})a_i+2^{-\nu+1}e^{t_i}\norm{\bar u}_{\mathbf{L}^1(\R)}\big]\cdot (x_2-x_1)
\end{equation*}
for all $x_1<x_2$. Hence \eqref{Pos-D} follows if we can show that
\begin{equation}\label{helpest}
(1+2^{-\nu})a_i+2^{-\nu+1}e^{t_i}\norm{\bar u }_{\mathbf{L}^1(\R)}~\leq ~\frac1{t_i}+2+2t_i+4e^{t_i}t_i\norm{\bar u}_{\mathbf{L}^1(\R)}.
\end{equation}
We therefore establish an upper bound for $\{a_i\}$. Observe first  that 
$$
\frac{1}{a_{i+1}}= 2^{-\nu}+\frac{1}{1+\Big(1+\frac{2e^{t_i}\norm{\bar{u}}_{\mathbf{L}^1(\R)}}{a_i}\Big)2^{-\nu}}\cdot \frac{1}{a_i}\,.
$$
Fix any $T>0$, set 
\[
K_T=1+2e^{T}\norm{\bar{u}}_{\mathbf{L}^1(\R)}\,,
\]
and define the sequence $\{z_i\}$ by
\bel{Se1}
z_1=2^{-\nu} \quad\text{and}\quad z_{i+1}=2^{-\nu}+\frac{1}{1+(1+K_T z_i) 2^{-\nu}}\cdot z_i\quad\text{for all } 1\leq i+1\leq T\cdot  2^{\nu}\,.
\eeq
By a comparison argument, one has
\bel{z<a-1}
z_{i+1}\leq\frac{1}{a_{i+1}},\qquad\text{for all }1\leq i+1\leq T \cdot 2^{\nu}\,.
\eeq
On the other hand, since 
\[
z_{i+1}\leq2^{-\nu}+z_i\qquad\text{for all } 1\leq i+1\leq T \cdot 2^{\nu}
\]
it holds that
\[
z_{i+1}~\leq~i\cdot 2^{-\nu}+z_1~=~(i+1)\cdot 2^{-\nu}~\leq~ T, \quad\text{for all } 1\leq i+1\leq T\cdot 2^{\nu}\,.
\]
Recalling (\ref{Se1}) we get 
\[
z_{i+1}~\geq~2^{-\nu}+\frac{1}{1+\big(1+K_T T\big)\cdot 2^{-\nu}}\cdot z_i~,\quad\text{for all } 1\leq i+1\leq T \cdot 2^{\nu}\,.
\]
Equivalently,
\[
z_{i+1}-\alpha~\geq~\frac{1}{1+\big(1+K_T T\big)\cdot 2^{-\nu}}\cdot (z_i-\alpha),\quad\text{for all } 1\leq i+1\leq T \cdot 2^{\nu}
\]
where $ \alpha:=2^{-\nu}+\frac{1}{1+K_TT}$. This implies that 
\begin{eqnarray*}
z_{i+1}-\alpha&\geq&\frac{1+\big(1+K_T T\big)\cdot 2^{-\nu}}{[1+\big(1+K_T  T\big)\cdot 2^{-\nu}]^{i+1}}\cdot  (z_1-\alpha)
\\
\\
&=&-\frac{1+\big(1+K_T T\big)\cdot 2^{-\nu}}{ [1+\big(1+K_T  T\big)\cdot2^{-\nu}]^{i+1}} \cdot \frac{1}{1+K_T T}
\\
\\
&\geq&-\frac{1+\big(1+K_T T\big)\cdot 2^{-\nu}}{1+ \big(1+K_T T\big)\cdot t_{i+1}}\cdot  \frac{1}{1+K_T T}
\end{eqnarray*}
for all $1\leq i+1\leq T \cdot 2^{\nu}$. Thus,
\begin{eqnarray*}
z_{i+1}&\geq&\frac{1}{1+K_T T}\cdot \Big[1-\frac{1}{1+\big(1+K_T T\big)\cdot t_{i+1}}\Big]
\cr\cr
& &\qquad\qquad \qquad\qquad+~2^{-\nu}\cdot \Big[1-\frac{1}{1+\big(1+K_T T\big)\cdot t_{i+1}}\Big]
\\
\\
&\geq&\frac{t_{i+1}}{1+\big(1+K_T T\big)\cdot t_{i+1}},\qquad\text{for all } 1\leq i+1\leq T\cdot 2^{\nu}\,.
\end{eqnarray*}
Recalling \eqref{z<a-1}, we have
\[
a_{i+1}~\leq~\frac{1}{z_{i+1}}~\leq~\frac{1}{t_{i+1}}+(1+K_T T),\qquad\text{for all } 1\leq i+1\leq T\cdot 2^{\nu}\,,
\]
and in particular, 
\[
a_{\lfloor T \cdot 2^{\nu}\rfloor +1}~\leq~\frac{1}{\lfloor T \cdot2^{\nu}\rfloor +1}+(1+K_T T)\,.
\]
Since the above inequality holds for any $T>0$, we obtain 
\[
a_{i}~\leq~\frac{1}{t_i}+1+t_i+2e^{t_i}t_i\norm{\bar{u}}_{\mathbf{L}^1(\R)},\qquad\text{for all } i\geq 1\,,
\]
which implies \eqref{helpest} and thus \eqref{Pos-D}.
\quad\\
\quad\\
{\bf 3. Minimal and maximal backward characteristics.}
Given some initial data $\bar u(x)$, we can split it into a positive and a negative part
\begin{equation*}
\bar u(x)=\max \{ \bar u(x), 0\}+\min\{\bar u(x), 0\}= \bar u^+(x)+\bar u^-(x).
\end{equation*}
Similarly we can split the source term for each $x\in\R$ into a positive and a negative part,
\begin{equation*}
Q^{\nu}(t_i,x)=[G_x*u^{\nu}(t_i)](x)=Q^{\nu,+}(t_i,x)+Q^{\nu,-}(t_i,x).
\end{equation*}
We then define the function $u^{\nu,+}(t)$ as follows
\begin{equation*}
u^{\nu,+}(t)=S_{t-t_i}^B(u^{\nu,+}(t_i)), \quad t\in[t_i,t_{i+1}[, 
\end{equation*}
\begin{equation*}
u^{\nu,+}(0)=\bar u^+, \quad u^{\nu,+}(t_i)=u^{\nu,+}(t_i-)+2^{-\nu} Q^{\nu,+}(t_i).
\end{equation*}
Similarly one defines the function $u^{\nu,-}(t)$ as follows
\begin{equation*}
u^{\nu,-}(t)=S_{t-t_i}^B (u^{\nu,-}(t_i)), \quad t\in[t_i,t_{i+1}[, 
\end{equation*}
\begin{equation*}
u^{\nu,-}(0)=\bar u^-, \quad u^{\nu,-}(t_i)=u^{\nu,-}(t_i-)+2^{-\nu} Q^{\nu,-}(t_i).
\end{equation*}
Here it should be noted that one has in general 
\begin{equation*}
u^{\nu}(t,x)\not = u^{\nu,+}(t,x)+u^{\nu, -}(t,x).
\end{equation*}
However, one has 
\begin{align*}
u^{\nu,-}(t,x)~\leq~0~\leq~u^{\nu,+}(t,x), &\qquad  \norm{u^{\nu,\pm}(t,\cdot)}_{\mathbf{L}^1(\R)}~\leq~e^t \norm{\bar u}_{\mathbf{L}^1(\R)}\\
u^{\nu,-}(t,x)&~ \leq ~ u^{\nu}(t,x)~\leq ~u^{\nu,+}(t,x),\\
\norm{Q^{\nu,\pm}(t_i,\cdot)}_{\mathbf{L}^1(\R)}& ~\leq ~\norm{u^\nu(t_i,\cdot)}_{\mathbf{L}^1(\R)} ~\leq ~ e^{t_i}\norm{\bar u}_{\mathbf{L}^1(\R)}.
\end{align*}
\quad\\
\noindent
Denote by $t\mapsto x(t)$ the generalized characteristic to the approximating solution $u^{\nu}(t,x)$ through the point $(\tau, x(\tau))$.
In addition, let $t\mapsto y(t)$ be the minimal backward characteristic, i.e. the characteristic for the positive solution $u^{\nu,+}(t,x)$ through the point $(\tau, x(\tau))$. Then $u^{\nu}(t,x)\leq u^{\nu,+}(t,x)$ implies that $y(t)\leq x(t)$ for all $t\in[0,\tau]$ and in particular $x(\tau)-x(t)\leq y(\tau)-y(t)$ for all $t\in[0,\tau]$. To estimate $y(\tau)-y(t)$, we compute
\begin{align*}
0&~ \leq~ \int_{y(\tau)}^\infty u^{\nu,+}(\tau,x)~dx\\
&~ \leq ~\int_{y(t)}^\infty u^{\nu,+}(t,x)dx +\sum_{t<t_i\leq \tau} \int_{y(t_i)}^\infty(u^{\nu,+}(t_i,x)-u^{\nu,+}(t_i-,x))~dx\\
&\qquad+ \int_{t}^\tau \frac12 u^{\nu,+,2}(s,y(s))-u^{\nu,+}(s,y(s))\dot y(s)~ds\\
& \leq \norm{ u^{\nu,+}(t,.)}_{\mathbf{L}^1(\R)}+ \sum_{t<t_i\leq \tau} \int_{y(t_i)}^\infty 2^{-\nu}Q^{\nu,+}(t_i,x)dx-\frac12 \int_t^\tau \dot y(s)^2~ds\\
& ~\leq~ e^\tau \norm{\bar u}_{\mathbf{L}^1(\R)}+\sum_{t<t_i\leq \tau} 2^{-\nu}\norm{ Q^{\nu,+}(t_i,\cdot)}_{\mathbf{L}^1(\R)}-\frac12 \int_t^\tau \dot y(s)^2~ds\\
& ~\leq~ e^\tau \norm{\bar u}_{\mathbf{L}^1(\R)}+\sum_{t<t_i\leq \tau} 2^{-\nu} e^{t_i}\norm{\bar u}_{\mathbf{L}^1(\R)}-\frac 12 \int_t^\tau \dot y(s)^2~ds\\
& ~\leq~ (1+\tau)e^\tau \norm{\bar u}_{\mathbf{L}^1(\R)}-\frac12 \int_t^\tau \dot y(s)^2ds\,.
\end{align*}
Applying the Cauchy Schwarz inequality then yields
\begin{multline*}
0~\leq~ y(\tau)-y(t)~=~\int_{t}^{\tau} \dot y(s)ds\\~\leq~ (\tau-t)^{1/2}\left(\int_t^\tau \dot y(s)^2 ds\right)^{1/2}\leq \sqrt{2(1+\tau)e^\tau \norm{\bar u}_{\mathbf{L}^1(\R)}}\cdot (\tau-t)^{1/2}.
\end{multline*}
\quad\\
\noindent 
Denote by $t\mapsto \tilde y(t)$ the maximal backward characteristic, i.e. the characteristic for the negative solution $u^{\nu,-}(t,x)$ through the point $(\tau, x(\tau))$. Then $u^{\nu,-}(t,x)\leq u^{\nu}(t,x)$ implies that $x(t)\leq \tilde y(t)$ for all $t\in[0,\tau]$, and in particular $\tilde y(\tau)-\tilde y(t)\leq x(\tau)-x(t)$ for all $t\in[0,\tau]$. A similar argument as before shows that 
\begin{equation*}
0\leq \tilde y(t)-\tilde y(\tau)\leq \sqrt{2(1+\tau)e^\tau \norm{\bar u}_{\mathbf{L}^1(\R)}}\cdot (\tau-t)^{1/2}.
\end{equation*}

Since 
\begin{equation*}
 \tilde y(\tau)-\tilde y(t)~\leq ~x(\tau)-x(t)\leq y(\tau)-y(t),
\end{equation*}
we have shown the following lemma.

\begin{lemma}\label{lem:charac}
For any $\nu\geq 1$, let $t\mapsto x(t)$ be any characteristic for the approximate solution $u^{\nu}(t,x)$. Then
\begin{equation}
\vert x(\tau)-x(t)\vert ~\leq ~C_1\cdot (\tau-t)^{1/2} \quad \text{ for all }0\leq t<\tau\leq T\,,
\end{equation}
where $C_1=\sqrt{2(1+T)e^T \norm{\bar u}_{\mathbf{L}^1(\R)}}$.
\end{lemma}
\quad\\

\noindent
{\bf 4. Lipschitz type estimate with respect to time.} We claim that for any fixed $\delta,R, T>0$ there exist constants $C_{1,\delta},C_{2,\delta}>0$ such that
\bel{Lip-C}
\norm{u^{\nu}(t,\cdot)-u^{\nu}{(s,\cdot)}}_{\mathbf{L}^1([-R,R])}~\leq ~C_{1,\delta}\cdot |t-s|+C_{2,\delta}\cdot 2^{-\nu}
\eeq
for all $t_1\leq\delta\leq s\leq t\leq T$.\\
\quad
\\
Due to \eqref{unudef} and \eqref{L-1G}, we have 
\bel{L-d}
\norm{ u^{\nu}(t_i,\cdot)-u^{\nu}(t_i-,\cdot)}_{\mathbf{L}^1([-R,R])}~\leq~ 2^{-\nu} e^{t_i}\norm{ \bar u}_{\mathbf{L}^1(\R)}\,.
\eeq
On the other hand, for any $s$, $t\in [t_i,t_{i+1}[$, we have following \cite[Theorem 7.10]{HR},
\begin{multline*}
\norm{ u^{\nu}(t,\cdot)-u^{\nu}(s,\cdot)}_{\mathbf{L}^1([-R,R])}~=~\norm{ S_{t-s}^B (u^{\nu}(s))(\cdot)-u^{\nu}(s,\cdot)}_{\mathbf{L}^1([-R,R])}\\
\\
~\leq~\max_{t\in [t_{i}, t_{i+1}]} \big\{\text{Tot.Var.}\{u^\nu (t,\cdot);[-R,R]\}\big\} \max_{t\in[t_{i},t_{i+1}]}\norm{u^{\nu}(t,\cdot)}_{\mathbf{L}^\infty([-R,R])} \vert t-s\vert.
\end{multline*}
\quad\\
\noindent
We are going to establish an upper bound for $\norm{ u^{\nu}(t,\cdot)}_{\mathbf{L}^\infty(\R)}$ and $\text{Tot.Var.} \{u^\nu (t,.);[-R,R]\}$ for $t\in [t_i,t_{i+1}[$. Let $b_i=\frac1{t_i}+2+2t_i+4t_ie^{t_i}\norm{\bar u}_{\mathbf{L}^1(\R)}$.
Combining \eqref{L-1-A}, \eqref{Pos-D}, and \eqref{L-I-A} then yields
\begin{equation*}
\norm{u^\nu(t,\cdot)}_{\mathbf {L}^\infty(\R)}~\leq~\sqrt{2 b_i \norm{u^\nu(t,.)}_{\mathbf{L}^1(\R)}}~\leq~\sqrt{2 
b_i e^t\norm{\bar u}_{\mathbf{L}^1(\R)}} \quad \text{ for all } t\in[t_i,t_{i+1}[.
\end{equation*}
Thus it is left to establish an upper bound for the total variation. Observe first that \eqref{Pos-D} implies that 
the function $u^{\nu}(t,x)-b_i x$ is decreasing. Then we have, for $t\in[t_i,t_{i+1}[$, 
\begin{align*}
\text{Tot.Var.}\{u^{\nu}(t,\cdot);[-R,R]\}&~\leq~ \text{Tot.Var.}\{u^{\nu}(t,\cdot)-b_i \cdot ;[-R,R]\}+\text{Tot.Var.}\{b_i \cdot :[-R,R]\}\\
&~=~u^{\nu}(t, -R)+b_i R-u^{\nu}(t,R)+b_i R+2b_i R\\
&~=~ u^{\nu}(t,-R)-u^{\nu}(t,R)+4b_i R\\
&~\leq~ 2\sqrt{2 b_i e^t\norm{\bar u}_{\mathbf{L}^1(\R)}}+4b_i R.
\end{align*}
Thus for all $s,t\in [t_i,t_{i+1}[$,
\begin{equation}\label{L-e}
\norm{u^\nu(t,\cdot)-u^\nu(s,\cdot)}_{\mathbf{L}^1([-R,R])}\leq (6b_ie^t\norm{\bar u}_{\mathbf{L}^1(\R)}+4b_i^2R^2)\vert t-s\vert.
\end{equation}
\noindent
Combining \eqref{L-d} and \eqref{L-e} finally yields for $t_1\leq \delta\leq s\leq t\leq T$,
\begin{multline*}
\|u^{\nu}(t,\cdot)-u^{\nu}(s,\cdot)\|_{\mathbf{L}^1([-R,R])}\\
 ~\leq~ \Big[6\left(\frac{1}{\lfloor \delta\cdot 2^{\nu}\rfloor \cdot 2^{-\nu}}+2+2T+4Te^T\norm{\bar u}_{\mathbf{L}^1(\R)}\right) e^T\norm{\bar u}_{\mathbf{L}^1(\R)}\\
\qquad +4\left(\frac{1}{\lfloor \delta\cdot 2^{\nu}\rfloor \cdot 2^{-\nu}}+2+2T+4Te^T\norm{\bar u}_{\mathbf{L}^1(\R)}\right)^2 R^2\\
+e^T\norm{\bar u}_{\mathbf{L}^1(\R)}\Big]\cdot \vert t-s\vert + 2^{-\nu}e^T\norm{\bar u}_{\mathbf{L}^1(\R)}.
\end{multline*}
\quad\\
\noindent 
{\bf 5. Tightness property.} We are establishing a {\it Tightness Property} for the sequence $u^{\nu}(t,x)$. Namely, given $\ve>0$ and $T>0$, there exists $L(T)>0$ such that 
\bel{TN}
\int_{|x|>L(T)}|u^{\nu}(t,z)|~dz~\leq~\ve\qquad\text{for all } t\in [0,T[~,\nu\geq 1\,.
\eeq

\quad\\
\noindent
To prove \eqref{TN} we are going to use a comparison argument. Given $\bar u\in L^1(\R)$, let $C_T=\sqrt{T}\cdot\sqrt{2(1+T)e^T\norm{ \bar u}_{L^1(\R)}}$ and consider any approximating solution $u^{\nu}(t,x)$ constructed by the flux splitting. By induction we define the sequence of radii $(R_i)_{i\geq 1}$ as follows.
\begin{enumerate}
\item  The radius $R_1$ is chosen so that 
\begin{equation*}
\int_{|x|\geq R_1-C_T} |\bar u(x)| dx~\leq~\frac12.
\end{equation*}
\item If $R_{i-1}$ is given, we choose $R_i$ in such a way that 
\begin{equation*}
\int_{|x|\geq R_i-C_T} |\bar u(x)|dx~\leq ~2^{-i}
\end{equation*}
and 
\begin{equation*}
R_i-R_{i-1}~\geq~(i+1)\ln (2)+2\sqrt{T}\sqrt{2(1+T)e^T\norm{\bar u}_{\mathbf{L}^1(\R)}}
\end{equation*}
\end{enumerate}
\quad\\
\noindent
Given the approximating solution $u^{\nu}(t,x)$, we denote by $R_i^{\pm}(t)$ the minimal and maximal backward characteristics through the point $(t,x)=(T,\pm R_i)$. For each $t\in [0,T]$, we define
\begin{equation*}
H_0(t)~=~\{ u\in\mathbf{L}^1(\R)\mid Supp(u)\subset [R_1^-(t), R_1^+(t)]\}
\end{equation*}
and 
\begin{equation*}
H_i(t)~=~\{ u\in\mathbf{L}^1(\R) \mid Supp(u)\subset [R_{i+1}^-(t), R_{i}^-(t)]\cup [R_{i}^+(t), R_{i+1}^+(t)]\} \quad \text{ for } i\geq 1. 
\end{equation*}
Furthermore for each $i\geq 1$, let
\[ K_i^-(t)~\doteq~H_0(t)\oplus H_1(t)\oplus H_2(t)\oplus\cdots\oplus H_{i-1}(t)\,,\qquad
K_i^+(t)~\doteq~H_i(t)\oplus H_{i+1}(t)\oplus\cdots \] 
with orthogonal projections
\[ \pi_i^-(t):\mathbf{L}^1(\R)\mapsto K_i^-(t)\,,\qquad \pi_i^-(t)(u)~=~\begin{cases} u(z), & \quad z\in[R_{i}^-(t), R_{i}^+(t)], \\
0, & \quad \text{ else,}
\end{cases}
\]
and 
\[\pi_i^+(t):{\bf L}^1(\R)\mapsto K_i^+(t)\,, \qquad \pi_i^+(t)(u)~=~\begin{cases} u(z), & \quad z\not\in [R_{i}^-(t), R_{i}^+(t)],\\
0, & \quad \text{ else.}
\end{cases}
\]
Then $K_i^-(t)\oplus K_i^+(t)=\mathbf{L}^1(\R)$ for all $i\geq 1$.
\noindent

Let $a_1(t)=\norm{\bar u}_{\mathbf{L}^1(\R)}e^t$ and define $a_i(t)$ for $i\geq 2$ inductively as the solution to 
\begin{equation*}
\frac{d}{dt}a_i(t)~=~a_{i-1}(t)+2^{-i}\norm{\bar u}_{\mathbf{L}^1(\R)}e^t \quad \text{and} \quad a_i(0)~=~2^{-i}.
\end{equation*}
Then $a_i(t)$ is non-decreasing. Moreover, $A(t)=\sum_{i\geq 1} a_i(t)$ solves
\begin{equation*}
\frac{d}{dt} A(t) ~=~ A(t)+\frac32 \norm{\bar u}_{\mathbf{L}^1(\R)}e^t \quad \text{ and }\quad A(0)=~~\frac12+\norm{\bar u}_{\mathbf{L}^1(\R)}.
\end{equation*}
Thus $A(t)=(\frac12+\norm{\bar u}_{\mathbf{L}^1(\R)}+\frac32\norm{\bar u}_{\mathbf{L}^1(\R)}t)e^t$, and hence to each $\varepsilon, T >0$ there exists $I\geq 1$ such that $a_i(t)\leq \varepsilon$ for all $t\in [0,T]$ and $i\geq I$.
\quad\\
\quad\\
\noindent
Hence if we can show that $p_i(t)=\norm{ \pi_i^+ u^{\nu}(t)}_{\mathbf{L}^1(\R)}$ ($i\geq 1$) satisfies 
\begin{equation}\label{case1}
p_1(t)~\leq ~\norm{\bar u}_{\mathbf{L}^1(\R)}e^t 
\end{equation} 
and
\begin{equation*}
p_i(0)~\leq~ 2^{-i} \quad \text{ and } \quad  p_i(t)~\leq~a_i(t) \text{ for all } i\geq 2, t\in [0,T]\,,
\end{equation*}
the claim follows. 
As far as $p_1(t)$ is concerned we have 
\begin{equation*}
p_1(t)~=~\norm{ u^{\nu}(t)}_{\mathbf{L}^1(\R\backslash [R_1^-(t), R_1^+(t)])}~\leq~ \norm{u^{\nu}(t)}_{\mathbf{L}^1(\R)}~\leq~ \norm{\bar u}_{\mathbf{L}^1(\R)}e^t
\end{equation*}
which proves \eqref{case1}.
By construction we have for $i\geq 2$, 
\begin{equation*}
p_i(0)~=~\int_{\vert x\vert \geq R_i(0)} \vert \bar u(x)\vert~dx~ \leq ~\int_{\vert x\vert~\geq~  R_i-C_T} \vert \bar u(x)\vert~dx ~\leq~ 2^{-i}
\end{equation*}
due to Lemma~\ref{lem:charac}.
Since the curves $R_i^-(t)$ and $R_i^+(t)$ are characteristics, we have 
\begin{equation}\label{estdec}
\frac{d}{dt}~p_i(t)~\leq~0 \quad \text{ for a.e. } t\in[t_{j-1},t_j[.
\end{equation}
On the other hand, for $t_j=j\cdot 2^{\nu}$, we have 
\begin{multline*}
\vert p_i(t_j)-p_i(t_j-)\vert ~\leq~ 2^{-\nu} \norm{ [G_x*u^{\nu}(t_j-,\cdot)]}_{\mathbf{L}^1(\R\backslash[R_i^-(t), R_i^+(t)])}\\
~=~ 2^{-\nu}\norm{ [G_x*(\pi_{i-1}^- u^{\nu}(t_j-)+\pi_{i-1}^+u^{\nu}(t_j-))](\cdot)}_{\mathbf{L}^1(\R\backslash[R_i^-(t), R_i^+(t)])}\\
~\leq~ 2^{-\nu} \norm{\pi_{i-1}^+ u^{\nu}(t_j-,.)}_{\mathbf{L}^1(\R)}+2^{-\nu}\norm{[G_x*\pi_{i-1}^- u^{\nu}(t_j-)]}_{\mathbf{L}^1(\R\backslash [R_i^-(t), R_i^+(t)])}.
\end{multline*}
As far as the first term on the right hand side is concerned we can apply \eqref{estdec}. The second one on the other hand is a bit more challenging,
\begin{multline*}
\| [G_x *\pi_{i-1}^- u^{\nu}(t_j-)] \|_{L^1(\R\backslash [R_i^-(t), R_i^+(t)])}\\
 \leq~ \int_{-\infty}^{R_i^-(t)}\int_{R_{i-1}^-(t)}^{R_{i-1}^+(t)} e^{-\vert x-y\vert }\vert u^{\nu}(t_j-,y)\vert~dy~dx\\
\qquad\qquad\qquad +\int_{R_i^+(t)}^\infty \int_{R_{i-1}^-(t)}^{R_{i-1}^+(t)} e^{-\vert x-y\vert } \vert u^{\nu}(t_j-,y)\vert~dy~dx\\
 \leq~ \int_{R_{i-1}^-(t)}^{R_{i-1}^+(t)} \left(\int_{-\infty}^{R_i^-(t)} e^{-(y-x)} dx +\int_{R_{i}^+(t)}^\infty e^{-(x-y)} dx\right) \vert u^{\nu}(t_j-,y)\vert~dy\\
\qquad\qquad \leq~ \int_{R_{i-1}^-(t)}^{R_{i-1}^+(t)}\Big(e^{-(R_{i-1}^-(t)-R_i^-(t))}+e^{-(R_i^+(t)-R_{i-1}^+(t))}\Big)\vert u^{\nu}(t_j-,y)\vert~dy\\
 \leq~\norm{u^{\nu}(t_j-,\cdot)}_{\mathbf{L}^1(\R)} \Big(e^{-(R_{i-1}^-(t)-R_i-(t))}+e^{-(R_i^+(t)-R_{i-1}^+(t))}\Big)\,.
\end{multline*}
\quad\\
\noindent
Now we can use the estimates for the minimal and maximal backward characteristics to conclude the proof.
We have 
\begin{equation*}
-\big(R_{i}^+(t)-R_{i-1}^+(t)\big)~\leq~2\sqrt{T}\sqrt{2(1+T)e^T\norm{\bar u}_{L^1(\R)}}-\big(R_i-R_{i-1}\big)
\end{equation*}
and 
\begin{equation*}
-\big(R_{i-1}^-(t)-R_i^-(t)\big)~\leq~ 2\sqrt{T}\sqrt{2(1+T)e^T \norm{\bar u}_{L^1(\R)}}-\big(R_i-R_{i-1}\big).
\end{equation*}
Thus 
\begin{align*}
\norm{ [G_x*\pi_{i-1}^- u^{\nu}(t_j-)]}_{\mathbf{L}^1(\R\backslash [R_i^-(t), R_i^+(t)])}
&~\leq~ 2^{-i}\cdot \norm{u^{\nu}(t_j-,.)}_{\mathbf{L}^1(\R)} \\  
&~\leq~ 2^{-i}\cdot \norm{u^{\nu}(t_{j-1},.)}_{\mathbf{L}^1(\R)}\\
&~\leq~ 2^{-i}\cdot e^{t_{j-1}} \norm{\bar u}_{\mathbf{L}^1(\R)}
\end{align*}
according to \eqref{unudef} and \eqref{L-1B},
and 
\begin{align*}
\vert p_i(t_j)-p_i(t_j-)\vert &~\leq~ 2^{-\nu}\cdot [p_{i-1}(t_{j-1})+2^{-i}e^{t_{j-1}}\norm{\bar u}_{\mathbf{L}^1(\R)}]\\
&~\leq ~2^{-\nu}\cdot [a_{i-1}(t_{j-1})+2^{-i} e^{t_{j-1}} \norm{\bar u}_{\mathbf{L}^1(\R)}]\\
&~\leq~ \int_{t_{j-1}}^{t_j}  a_{i-1}(t)+2^{-i}e^t\norm{\bar u}_{\mathbf{L}^1(\R)} dt.
\end{align*}
This means in particular, together with \eqref{estdec}, that 
\[
\int_{\big\{x<R_i^-(t)\}\cup \big\{x>R_i^+(t)\big\}} \vert u^\nu(t,x)\vert dx ~\leq~ a_i(t).
\]
However, Lemma~\ref{lem:charac} yields 
\begin{equation*}
\vert R_i^+(t)-R_i\vert~\leq~ C_T\quad\mathrm{and}\quad \vert R_i^-(t)+R_i\vert ~\leq~ C_T \quad \text{ for all } t\in [0,T], i\geq 1.
\end{equation*}
Given $\varepsilon >0$ and $T>0$, we choose $i$ such that $a_i(T)\leq \varepsilon$ and hence $a_i(t)\leq \varepsilon$ for all $t\in [0,T]$. Choosing $L(T)=R_i+C_T$ finishes the proof of \eqref{TN}. 

\section{Existence and uniqueness of weak entropy solutions}
\label{sec:3}

After introducing the approximating sequence $\{u^{\nu}(t,x)\}_{\nu\in\mathbb{N}}$ and deriving some a priori estimates in the last section, we are going to establish the existence and uniqueness of weak entropy solutions, i.e. Theorem~\ref{thm1}. 

\begin{proof}[Proof of Theorem~\ref{thm1}] 
\quad\\
\quad\\
\noindent
{\bf 1. Existence of a limiting function.} Let $\{u^{\nu}(t,x)\}_{\nu\in\mathbb{N}}$ be the sequence of  approximating solutions constructed in Section~\ref{sec:2}. In addition, we introduce a new sequence $\{ \tilde u^\nu(t,x)\}_{\nu\in\mathbb{N}}$, by defining
\begin{equation}\label{defutilde}
\tilde u^\nu(t,\cdot):= (1-\theta_t)\cdot u^\nu(t_i, \cdot)+\theta_t\cdot u^\nu(t_{i+1},\cdot)\, \quad \text{for all } t\in[t_i,t_{i+1}[,
\end{equation}
where $\theta_t\in[0,1[$ such that $t=(1-\theta_t)\cdot t_i+\theta_t \cdot t_{i+1}$.
In contrast to $u^{\nu}(t,x)$ the function $\tilde u^{\nu}(t,x)$ satisfies 
\[
\tilde u^{\nu}(t_i,x)~=~\tilde u^{\nu}(t_i-,x), \quad \text{ for all }i=1,2,\dots,
\]
 a property which plays a crucial role in establishing the existence of a convergent subsequence. 
To this end we are going to apply \cite[Theorem A.8]{HR}, which we state here, in a slightly modified version, for the sake of completeness.
\begin{theorem}\label{thmsubseq}
Let $\tilde u^{\nu}:[\delta, \infty)\times\R\to \R$ be a family of functions such that for each $T>\delta$, 
\bel{A83}
\vert \tilde u^{\nu}(t,x)\vert~\leq~C_T, \qquad (t,x)\in [\delta, T]\times \R\,,
\eeq
for a constant $C_T$ independent of $\nu$. Assume in addition for all compact $B\subset \R$ and for $t\in [\delta,T]$ that 
\bel{A81}
\sup_{\vert \xi\vert \leq \vert \rho\vert }\int_B \vert \tilde u^{\nu}(t,x+\xi)-\tilde u^{\nu}(t,x)\vert dx~\leq~v_{B,T}(\vert \rho\vert),
\eeq
for a modulus of continuity $v_{B,T}$. Furthermore, assume for $s$ and $t$ in $[\delta,T]$ that 
\bel{A82}
\int_B \vert \tilde u^{\nu}(t,x)- \tilde u^{\nu}(s,x)\vert dx~\leq~w_{B,T}(\vert t-s\vert) \quad \text{as } \nu\to \infty,
\eeq
for some modulus of continuity $w_{B,T}$. Then there exists a sequence $\nu_j\to\infty$ such that for each $t\in [\delta,T]$ the sequence $\{\tilde u^{\nu_j}(t,\cdot)\}$ converges to a function $u(t,\cdot)$ in $\mathbf{L}^1_{loc}(\R)$. The convergence is in $C([\delta,T]; \mathbf{L}^1_{loc}(\R))$.
\end{theorem}

We start by checking that all the assumptions in the above theorem are fulfilled for the sequence $\{\tilde u^{\nu}(t,x)\}$. Without loss of generality, we assume that $\nu$ satisfies $2^{-\nu}\leq \delta\leq 2\cdot \lfloor  2^{\nu}\cdot \delta\rfloor \cdot 2^{-\nu}$.

\eqref{A82}: It suffices to show that for any fixed $\delta, R, T>0$, there exists a constant $L_{\delta}$ such that 
\bel{bd2}
\norm{\tilde{u}^{\nu}(t,\cdot)-\tilde{u}^{\nu}(s,\cdot)}_{L^1([-R,R])}\leq L_{\delta} |t-s|,\qquad \delta\leq s\leq t \leq T\,.
\eeq
Let $s,t\in [\delta, T]$ such that $s=(1-\theta_s)t_i+\theta_s t_{i+1}$ and $t=(1-\theta_t)t_i+\theta_t t_{i+1}$, then $\vert t-s\vert = 2^{-\nu}\vert \theta_t-\theta_s\vert$ and 
\begin{align*}
\norm{\tilde u^{\nu}(t,\cdot)-\tilde u^{\nu}(s,\cdot)}_{L^1([-R,R])}&~\leq ~\vert \theta_t-\theta_s\vert \norm{ u^{\nu}(t_{i+1},\cdot)-u^{\nu}(t_i,\cdot)}_{L^1([-R,R])}\\
&~\leq~ \vert \theta_t-\theta_s\vert (C_{1,\delta}\vert t_{i+1}-t_i\vert +C_{2,\delta} 2^{-\nu})\\
&~= ~\vert \theta_t-\theta_s\vert 2^{-\nu} (C_{1,\delta}+C_{2,\delta})\\
& ~ =~ \vert t-s\vert (C_{1,\delta}+C_{2,\delta})
\end{align*}
where we used \eqref{Lip-C}.
In the general case, $s,t\in[\delta, T]$ with $s=(1-\theta_s)t_i+\theta_s t_{i+1}<t=(1-\theta_t)t_j+\theta_t t_{j+1}$ and $i\not = j$ one obtains 
\begin{align*}
\norm{\tilde u^{\nu}(t,\cdot)-\tilde{u}^{\nu}(s,\cdot)}_{\mathbf{L}^1([-R,R])}& ~\leq~ \norm{\tilde u^{\nu}(t,\cdot)-\tilde{u}^{\nu}(t_j,\cdot)}_{\mathbf{L}^1([-R,R])}\\
&\quad +~\sum_{i+1\leq k\leq j-1} \norm{ \tilde u^{\nu}(t_{k+1},\cdot)-\tilde u^{\nu}(t_k,\cdot)}_{\mathbf{L}^1([-R,R])}\\
& \quad +~\norm{ \tilde u^{\nu}(t_{i+1},\cdot)-\tilde u^{\nu}(s,\cdot)}_{\mathbf{L}^1([-R,R])}\\
& ~\leq~ \vert t-s\vert (C_{1,\delta}+C_{2,\delta}).
\end{align*}
Thus choosing $L_{\delta}=C_{1,\delta}+C_{2,\delta}$ yields \eqref{bd2} and hence \eqref{A82}.

\eqref{A83}: Observe that due to \eqref{Pos-D}, we have for $t=(1-\theta_t)t_i+\theta_t t_{i+1}$ and all $x<y$,
\begin{align}\label{TVest}
\tilde u^{\nu}(t,y)-\tilde u^{\nu}(t,x)
& \leq (1-\theta_t)(u^\nu(t_i,y)-u^\nu(t_i,x)) +\theta_t (u^\nu(t_{i+1},y)-u^\nu(t_{i+1},x))\\ \nn 
&~\leq~ \Big(\frac1{t_i} +2 +2t_{i+1}+4t_{i+1}e^{t_{i+1}}\norm{\bar u}_{\mathbf{L}^1(\R)}\Big)(y-x)\\ \nn
&~ \leq~\Big(\frac{2}{\delta}+2+2(T+1)+4(T+1)e^{T+1}\norm{\bar u}_{\mathbf{L}^1(\R)}\Big) (y-x)\\ \nn 
& ~=~D_T(y-x)
\end{align}
and hence, as in the proof of Lemma~\ref{lem2}, we obtain
\bel{TVest2}
\norm{\tilde u^{\nu}(t,x)}_{\mathbf{L}^\infty(\R)}~\leq~ \sqrt{2 \Big(\frac{2}{\delta}+2+2(T+1)+4(T+1)e^{T+1}\norm{\bar u}_{\mathbf{L}^1(\R)}\Big) e^{T+1}\norm{ \bar u}_{\mathbf{L}^1(\R)}}.
\eeq
Thus choosing $C_T=\sqrt{2 (\frac{2}{\delta}+2+2(T+1)+4(T+1)e^{T+1}\norm{\bar u}_{\mathbf{L}^1(\R)}) e^{T+1}\norm{ \bar u}_{\mathbf{L}^1(\R)}}$ finishes the proof of \eqref{A83}.
\quad \\ 
\eqref{A81}:
Following the proof of \cite[Lemma A.1]{HR} and applying \eqref{TVest} and \eqref{TVest2}, we get
\begin{multline*}
\int_{[-R,R]}\vert \tilde u^{\nu}(t,x+\xi)-\tilde u^{\nu}(t,x)\vert~dx~\leq~ \text{Tot.Var.} \{\tilde u^{\nu}(t,\cdot);[-R-\vert \xi\vert ,R+\vert \xi\vert]\}\cdot  \vert \xi\vert\\
~\leq~ \vert \xi\vert (\text{Tot.Var.} \{\tilde u^{\nu}(t,\cdot)-D_T\cdot; [-R-\vert \xi\vert,R+\vert \xi\vert]\}+\text{Tot.Var.} \{D_T\cdot; [-R-\vert \xi\vert,R+\vert \xi\vert]\})\\
~\leq~ \vert \xi\vert (2\norm{\tilde u^{\nu}(t,\cdot)}_{\mathbf{L}^\infty(\R)}+4D_T(R+\vert \xi\vert))
\end{multline*}
since $\tilde u^{\nu}-D_T$ is decreasing, due to \eqref{TVest}, and hence \eqref{A81} is satisfied.

Thus, Theorem~\ref{thmsubseq} implies that there exists a subsequence $\nu_j\to\infty$ and a limit function $\tilde u:[\delta, T]\times [-R,R]\to \R$ such that 
\bel{L1}
\lim_{j\to\infty}\norm{\tilde{u}^{\nu_j}(t,\cdot)-\tilde{u}(t,\cdot)}_{\mathbf{L}^1([-R,R])}~=~0,\qquad\text{ for all } \delta \leq t\leq T\,,
\eeq
and 
\bel{L1q}
\lim_{j\to\infty}\norm{\tilde{u}^{\nu_j}-\tilde{u}}_{\mathbf{L}^1([\delta,T]\times [-R,R])}~=~0\,.
\eeq

By construction, see \eqref{defutilde}, we have for any $\delta\leq t\leq T$, with $t=(1-\theta_t)t_i+\theta_tt_{i+1}$ that 
\begin{multline*}
\norm{u^{\nu_j}(t,\cdot)-\tilde{u}^{\nu_j}(t,\cdot)}_{\mathbf{L}^1([-R,R])}\\
\leq(1-\theta_t)\norm{u^{\nu_j}(t,\cdot)-u^{\nu_j}(t_i,\cdot)}_{\mathbf{L}^1([-R,R])}+\theta_t\norm{u^{\nu_j}(t,\cdot)-u^{\nu_j}(t_{i+1},\cdot)}_{\mathbf{L}^1([-R,R])}\,.
\end{multline*}
Recalling \eqref{Lip-C}, we obtain 
\bel{L1d}
\norm{u^{\nu_j}(t,\cdot)-\tilde{u}^{\nu_j}(t,\cdot)}_{\mathbf{L}^1([-R,R])}~\leq~ L_{\delta} 2^{-\nu}\,
\eeq
with $L_{\delta}~:=~ C_{1,\delta}+C_{2,\delta}$. Thus combining \eqref{L1}, \eqref{L1q} and \eqref{L1d}, we get
\bel{L2}
\lim_{j\to\infty}\norm{u^{\nu_j}(t,\cdot)-\tilde{u}(t,\cdot)}_{\mathbf{L}^1([-R,R])}~=~0,\qquad\text{for all } \delta \leq t\leq T\,,
\eeq
and
\bel{L2q}
\lim_{j\to\infty}\norm{u^{\nu_j}-\tilde{u}}_{\mathbf{L}^1([\delta,T]\times [-R,R])}~=~0\,.
\eeq
\quad\\
Since \eqref{L2} and \eqref{L2q} hold for any $\delta,T,R>0$,  there exists $I\subset \mathbb{N}$ and $u:[0,\infty)\times\R\to\R$ such that $\{u^{\nu}(t,\cdot)\}_{\nu\in I}\to u(t,\cdot)$ in $\mathbf{L}^1_{loc}(\R)$ for any $t\geq 0$ and $\{u^{\nu}\}_{\nu\in I}\to u$ in $\mathbf{L}^1_{loc}([0,\infty[\times \R)$. Moreover, by the {\it Tightness property}, we have that to any $\varepsilon>0$ there exists $R_\varepsilon>0$ such that 
\[
\int_{\R\backslash[-R_\varepsilon, R_\varepsilon]} \vert u^{\nu}(t,x)\vert~dx~\leq~ \varepsilon\quad \text{for all } t\in[0,T], \nu\in I.
\]
Thus we get for all $R\geq R_\varepsilon$ 
\begin{align*}
\int_{[-R,R]} \vert u(t,x)\vert~dx & ~=~\lim_{\nu\in I\to\infty}\int_{[-R,R]}~\vert u^{\nu}(t,x)\vert~dx \\
& ~\leq ~\lim_{\nu\in I\to\infty}~(\norm{ u^{\nu}(t,.)}_{\mathbf{L}^1(\R)}+ \varepsilon)\\
& ~\leq ~e^T\norm{\bar u}_{\mathbf{L}^1(\R)}+\varepsilon.
\end{align*}
Since the above estimate is uniform, we can conclude that $u(t,.)\in {\bf L}^1(\R)$ for all $t\geq0$. Let $t\in [0,T]$, then there exists $R^1_{\ve}>0$ such that 
\[
\int_{\R\backslash[-R^1_\varepsilon, R^1_\varepsilon]} \vert u(t,x)\vert~dx~\leq~ \varepsilon.
\]
Set $\overline{R}_{\ve}\doteq\max\{R_{\ve},R^1_{\ve}\}$, we then have 
\[
\int_{\R\backslash[-\overline{R}_\varepsilon, \overline{R}_\varepsilon]} \vert u(t,x)-u^{\nu}(t,x)\vert~dx~\leq~ 2\varepsilon\qquad\text{for all } \nu\in I\,,
\]
which implies
\[
\lim_{\nu\in I\to \infty}~\|u(t,\cdot)-u^{\nu}(t,\cdot)\|_{{\bf L}^1(\R)}~\leq~2\ve\,.
\]
Therefore, $u^{\nu}(t,\cdot)$ converges to $u(t,\cdot)$ in ${\bf L}^1(\R)$ for all $t\geq 0$. 
\quad\\
\quad\\
Recalling \eqref{L-1-A} and \eqref{Pos-D}, we have for all $t>0$,
\[
\norm{u(t,\cdot)}_{L^1(\R)}\leq e^{t} \norm{\bar{u}}_{\mathbf{L}^{1}(\R)}\,,
\]
and 
\bel{Pos-Du}
u(t,x_2)-u(t,x_1)~\leq~\Big[\frac{1}{t}+2+2 t+4te^{t}\norm{\bar{u}}_{\mathbf{L}^1(\R)}\Big](x_2-x_1),\quad\text{for all }x_1<x_2\,.
\eeq 
In particular,
\[
\|u(t,\cdot)\|_{{\bf L}^{\infty}(\R)}~\leq~\sqrt{2e^t\cdot \Big[\frac{1}{t}+2+2 t+4te^{t}\norm{\bar{u}}_{\mathbf{L}^1(\R)}\Big]\cdot \norm{\bar{u}}_{\mathbf{L}^{1}(\R)}},\quad\text{for all } t>0\,,
\]
and thus $u$ is in ${\bf L}^{\infty}_{loc}(]0,\infty[, {\bf L}^{\infty}(\R))$.
\\
\quad\\
\noindent
{\bf 2. The map $t\to u(t,\cdot)$ is continuous from $[0,T[$ to $\mathbf{L}^1(\R)$.} 
By the {\it Tightness Property} for $\{ u^\nu(t,x)\}_{\nu\in\mathbb{N}}$, we have that to any $\ve>0$, there exists $R_{\ve}>0$ such that 
\begin{equation}\label{L2QQ}
\norm{u^{\nu}(t,\cdot)}_{\mathbf{L}^1(\R\backslash [-R_{\ve}, R_{\ve}])}\leq \ve \qquad  \text{for all }t\in[0,T[, \nu\in \mathbb{N}.
\end{equation}
Thus for $R_\ve$ big enough \eqref{L2q} and \eqref{L2QQ} imply for $t\in(\delta,T)$ that 
\begin{align*}
\norm{u(t,\cdot)}_{\mathbf{L}^1(\R\backslash [-R_{\ve}, R_{\ve}])}&~=~ \norm{u(t,\cdot)}_{\mathbf{L}^1(\R)}-\norm{u(t,\cdot)}_{\mathbf{L}^1([-R_{\ve}, R_{\ve}])}\\
& ~\leq ~ \lim_{\nu\in I \to\infty} (\norm{u^{\nu}(t,.)}_{\mathbf{L}^1(\R)} +\varepsilon-\norm{u^{\nu}(t,.)}_{\mathbf{L}^1([-R_\ve, R_\ve])})\\
& ~ \leq ~\lim_{\nu\in I \to\infty} (\norm{u^{\nu}(t,.)}_{\mathbf{L}^1(\R\backslash[-R_{\ve},R_{\ve}])}+\varepsilon)\\
&~ \leq ~2\varepsilon.
\end{align*}
Therefore, for any fixed $\delta >0$ and for any $s,t\in(\delta, T)$ we obtain
\begin{eqnarray*}
\norm{u(t,\cdot)-u(s,\cdot)}_{\mathbf{L}^1(\R)}&\leq&\norm{u(t,\cdot)-u(s,\cdot)}_{\mathbf{L}^1([-R_\varepsilon,R_\varepsilon])}+2\ve\\
\\
&=&\lim_{\nu\in I \to\infty}\norm{u^{\nu}(t,\cdot)-u^{\nu}(s,\cdot)}_{\mathbf{L}^1([-R_\varepsilon,R_\varepsilon])}+2\ve\\
\\
&\leq&C_{1,\delta}\cdot |t-s|+2\ve\,,
\end{eqnarray*}
where we applied \eqref{Lip-C} in the last step. 
This implies that $u(t,\cdot)$ is continuous from $(0,T)$ to $\mathbf{L}^1(\R)$. 

On the other hand, the continuity also holds at $t=0$, i.e.,
\bel{limit-0}
\lim_{t\to 0^+} \norm{u(t,\cdot)-\bar{u}}_{\mathbf{L}^1(\R)}~=~0\,.
\eeq
Indeed, for any $\nu\in\mathbb{N}$ and $t>0$, we have
\begin{multline*}
\|S^B_{t-t_{i}}(u^{\nu}(t_i))(\cdot)-S^B_{t-t_{i-1}}(u^{\nu}(t_{i-1}))(\cdot)\|_{\mathbf{L}^1(\R)}\\
~=~\norm{S^B_{t-t_{i}}(u^{\nu}(t_i))(\cdot)-S^B_{t-t_{i}}(u^{\nu}(t_{i}-))(\cdot)}_{\mathbf{L}^1(\R)}\\
~\leq~\norm{u^{\nu}(t_{i},\cdot)-u^{\nu}(t_{i}-,\cdot)}_{L^1(\R)}\leq2^{-\nu}e^{t}\norm{\bar{u}}_{\mathbf{L}^1(\R)}\,
\end{multline*}
for all $i\in \{1,2,\dots, \lfloor 2^{\nu}\cdot t\rfloor\}$. Thus,
\begin{align*}
\norm{u^{\nu}(t,\cdot)-S^B_{t}(\bar{u})(\cdot)}_{\mathbf{L}^1(\R)}& ~\leq~\sum^{\lfloor 2^{\nu}\cdot t\rfloor\}}_{i=1}\norm{S^B_{t-t_{i}}(u^{\nu}(t_i))(\cdot)-S^B_{t-t_{i-1}}(u^{\nu}(t_{i-1}))(\cdot)}_{\mathbf{L}^1(\R)}\\ 
& ~\leq~ te^{t}\norm{\bar{u}}_{\mathbf{L}^1(\R)}\,.
\end{align*}
Since $u^{\nu}(t,\cdot)$ converges to $u(t,\cdot)$ in ${\bf L }^1(\R)$ for all $t\geq 0$, we obtain that
\[
\|u(t,\cdot)-S_t^{B}(\bar{u})(\cdot)\|_{{\bf L}^1(\R)}~\leq~te^{t}\norm{\bar{u}}_{\mathbf{L}^1(\R)}\,
\]
and in particular, 
\[
\lim_{t\to 0^+}~\norm{u(t,\cdot)-S^B_{t}(\bar{u})}_{\mathbf{L}^1(\R)}~=0\,.
\]
Therefore, \eqref{limit-0} follows from the continuity of Burgers semigroup $S^B_t$ at time $t=0$.
\quad\\
\quad\\
{\bf 3. Weak entropy condition.} We show that $u(t,x)$ satisfies the entropy condition (\ref{entropy}). Let $\eta(u)=\vert u-k\vert $ and $q(u)=\mathrm{sign}(u-k)\left( \frac{u^2}{2}-\frac{k^2}{2}\right)$. In addition define 
$\eta_\delta(u)=\sqrt{(u-k)^2+\delta^2}$ and denote by $q_\delta(u)$ the solution to 
\begin{equation}
q_\delta'(u)=\eta_\delta'(u)u \qquad q_\delta(k)=0.
\end{equation}
Then we get for any $\phi\in C_1^c(]0,\infty[\times \R, \R)$, since $u(t,x)$ is uniformly bounded on the support of $\phi$ according to \eqref{Pos-Du}, that 
\begin{align*}
 \iint [\vert u-k\vert \phi_t & + \mathrm{sign}(u-k)\left(\frac{u^2}{2}-\frac{k^2}{2}\right) \phi_x]dxdt\\ 
& = \iint [ \eta(u)\phi_t +q(u)\phi_x]dx dt \\ 
& = \lim_{\delta\to 0} \iint [\eta_\delta(u)\phi_t+q_\delta(u)\phi_x] dxdt\\ 
& = \lim_{\delta\to 0}\lim_{j\to\infty} \iint [\eta_\delta(u^{\nu_j})\phi_t+q_{\delta}(u^{\nu_j})\phi_x] dxdt \\
& = \lim_{\delta\to 0}\lim_{j\to\infty} \sum_{i=0}^\infty \int_{t_i}^{t_{i+1}} \int_\R [\eta_\delta(u^{\nu_j})\phi_t+q_{\delta}(u^{\nu_j})\phi_x] dxdt \\
&\geq \lim_{\delta\to 0} \lim_{j\to\infty} \sum_{i=0}^\infty \int_\R [\eta_\delta(u^{\nu_j}(t_{i+1}-,x))\phi(t_{i+1},x)-\eta_\delta(u^{\nu_j}(t_i+,x))\phi(t_i,x) ]dx\\
& =\lim_{\delta\to 0}\lim_{j\to\infty}\sum_{i=1}^\infty \int_\R [\eta_\delta (u^{\nu_j}(t_{i}-,x))-\eta_\delta(u^{\nu_j}(t_i+,x))] \phi(t_i,x) dx\\
& = \lim_{\delta\to 0} \iint \frac{-(u(t,x)-k)[G_x*u(t)](x)}{\sqrt{(u-k)^2+\delta^2}}\phi(t,x) dxdt \\ 
& =-\iint\mathrm{sign}(u(t,x)-k) [G_x*u(t)](x)\phi(t,x)dx dt.
\end{align*}
\quad\\
{\bf 4. Lipschitz continuity with respect to time.} Let $u_1,u_2$ be weak entropy solutions of \eqref{BP} with $u_1(0,\cdot)=\bar{u}_1(\cdot)$ and $u_2(0,\cdot)=\bar{u}_2(\cdot)$, respectively. We will prove that 
\bel{Stabi}
\norm{u_2(t,\cdot)-u_1(t,\cdot)}_{\mathbf{L}^1(\R)}\leq e^{t}\cdot \norm{\bar{u}_2-\bar{u}_1}_{\mathbf{L}^1(\R)},\quad\text{for all } t>0\,,
\eeq
which implies the uniqueness of the weak entropy solution to \eqref{BP}-\eqref{id}.

Since $t\to u_i(t,\cdot)$ is continuous with values in $L^1(\R)$, for every $\ve>0$ there exists $t_{\ve}>0$ such that 
\[
\norm{u_i(t_{\ve},\cdot)-\bar{u}_i(\cdot)}_{\mathbf{L}^1(\R)}\leq\ve\qquad i\in\{1,2\}\,.
\]
This implies that 
\bel{near 0}
\norm{u_2(t_{\ve},\cdot)-u_1(t_{\ve},\cdot)}_{\mathbf{L}^1(\R)}\leq2\ve+\norm{\bar{u}_2-\bar{u}_1}_{\mathbf{L}^1(\R)}\,.
\eeq
\quad\\
On the other hand, since $u_1(t,x)$ and $u_2(t,x)$ are weak entropy solutions, for any $T>0$ there exists $M_T>0$ such that for any $ t\in [t_{\ve},T]$ it holds
\[
\norm{u_i}_{\mathbf{L}^{\infty}([t_{\ve},T]\times \R)}\leq M_T,\qquad i\in \{1,2\}\,.
\]
Therefore, one can follow the argument in the proof of \cite[Theorem 6.2]{B} to show that for all $t_{\ve}\leq s\leq t\leq T$
\begin{multline*}
\norm{u_2(t,\cdot)-u_{1}(t,\cdot)}_{\mathbf{L}^1(\R)}\\~\leq~ \norm{u_2(s,\cdot)-u_{1}(s,\cdot)}_{\mathbf{L}^1(\R)}+\int^t_{s} \norm{G_x*(u_2(\tau,\cdot)-u_1(\tau,\cdot))}_{\mathbf{L}^1(\R)} d\tau\\
\\
~\leq~ \norm{u_2(s,\cdot)-u_{1}(s,\cdot)}_{\mathbf{L}^1(\R)}+\int_s^t\norm{u_2(\tau,\cdot)-u_1(\tau,\cdot)}_{\mathbf{L}^1(\R)}d\tau\,.
\end{multline*}
Thus, the function $Z(t):= \norm{u_2(t,\cdot)-u_1(t,\cdot)}_{\mathbf{L}^1(\R)}$ satisfies the integral inequality 
\[
Z(t)\leq Z(t_{\ve})+\int^t_{t_\ve}Z(\tau) d\tau,\qquad\text{for all } t_{\ve}\leq t\leq T\,.
\]
Applying Gronwall's inequality, we get
\[
Z(t)\leq e^{(t-t_{\ve})} Z(t_{\ve}),\qquad\text{for all } t_{\ve}\leq t\leq T\,.
\]
Recalling \eqref{near 0}, we finally obtain that 
\[
Z(t)\leq e^{(t-t_{\ve})} \big[Z(0)+2\ve\big] \leq e^t  [Z(0)+2\ve]\qquad\text{for all } \ve>0\,,
\]
which yields \eqref{Stabi}\,.
\end{proof}

\section{Local smoothness and wave breaking}

In this final section we want to focus on the prediction of wave breaking. In particular, we are interested in identifying for initial data $\bar u(\cdot)\in \mathbf{C}^1(\R)\cap {\bf L}^1(\R)$ if wave breaking occurs in the nearby future or not by following solutions along characteristics as long as they exist in the classical sense.

\begin{theorem}\label{WB}
Let $u(t,x)$ be the weak entropy solution of \eqref{BP} with $u(0,\cdot)=\bar u(\cdot)\in \mathbf{C}^1(\R)\cap {\bf L}^1(\R)$. Denote by $x(t)$ the characteristic through $\bar{x}$, i.e., 
$$x'(t)=u(t,x(t)) \quad \mathrm\quad x(0)=\bar{x}\,.$$
Then the following statements hold.
\begin{itemize}
\item [(i)] $u_x(t,x(t))$ remains bounded for all $t\in [0,T_*[$ where 
\bel{lowb}
T_*~=~\ln\left(1+\frac{1}{\sqrt{|\bar u(\bar x)|+2\norm{\bar u}_{{\bf L}^1(\R)}}}\left( \frac{\pi}{2}+\arctan\left(\frac{\bar u'(\bar x)}{\sqrt{\vert \bar u(\bar x)\vert+2\norm{\bar u}_{{\bf L}^1(\R)}}}\right)\right)\right)\,.
\eeq
\item [(ii)] If
\bel{condblowup}
\bar u'(\bar{x})~<~-\frac{1}{2}-\sqrt{\vert \bar u(\bar{x})\vert +2\norm{\bar u}_{\mathbf{L}^1(\R)}+\frac14}\,,
\eeq
then $u_x(t,x(t))$ becomes unbounded before time $T^*$ where
\bel{upbound}
T^*~=~ \frac{2}{\Big|2\bar u'(\bar{x})+1+2\cdot \sqrt{\vert \bar u(\bar x)\vert +2\norm{\bar u}_{{\bf L}^1(\R)}+\frac14} \Big|}\,.
\eeq
\end{itemize}
\end{theorem}

\begin{proof}
Let $u(t,x)$ be the weak entropy solution of \eqref{BP} with initial data $u(0,x)=\bar u(x)\in\mathbf{C}^1(\R)\cap {\bf L}^1(\R) $. Given $\bar x\in\R$ denote by $x(t)$ the characteristic through $\bar{x}$ at time $t=0$, i.e., $x'(t)=u(t,x(t))$ and $x(0)=\bar x$. Furthermore let $m=\bar u'(\bar{x})$, then we are interested in finding an upper and a lower bound on how long it takes until wave breaking occurs. That is we are going to establish an upper and a lower bound on $t^*$ such that  
\[ 
u_x(t,x(t))\to -\infty \qquad t\uparrow t^*.
\]
Differentiating \eqref{BP} with respect to $x$ yields 
\[ 
(u_x)_t+u(u_x)_x~=~-(u_x)^2+[G*u]_{xx}~=~-(u_x)^2+[G*u]+u
\]
and hence 
\[ 
z(t):=u_x(t,x(t))
\]
satisfies 
\bel{along-char}
z'(t)~=~-z(t)^2+[G*u(t,\cdot)](x(t))+u(t,x(t)).
\eeq
Since $u(t,x(t))$ satisfies 
\[
\frac{d}{dt} u(t,x(t))~=~[G_x*u(t,\cdot)](x(t)),
\] 
we have, due to \eqref{L-1},
\[
\frac{d}{dt}\vert u(t,x(t))\vert ~\leq ~\norm{ u(t,.)}_{{\bf L}^1(\R)}~\leq~\norm{\bar u}_{{\bf L}^1(\R)}e^t \quad \text{for all } t\in[0,t^*).
\]
Thus 
\[ 
\vert u(t,x(t))\vert ~\leq ~ \vert \bar u(\bar x)\vert +\norm{\bar u}_{{\bf L}^1(\R)}e^t\leq \Big(\vert \bar u(\bar x)\vert +\norm{\bar u}_{{\bf L}^1(\R)}\Big)\cdot e^t, 
\] 
and 
\begin{multline*}
\vert [G*u(t,\cdot)](x(t))+u(t,x(t))\vert \\
\leq \norm{ u(t,\cdot)}_{{\bf L}^1(\R)}+\vert u(t,x(t))\vert ~\leq~ \Big(\vert \bar u(\bar x)\vert +2 \norm{\bar u}_{{\bf L}^1(\R)}\Big) \cdot e^t.
\end{multline*}
Recalling \eqref{along-char}, we get for all $t\in [0,t^*)$,
\[
-z(t)^2-\Big(\vert \bar u(\bar x)\vert +2\norm{\bar u}_{{\bf L}^1(\R)}\Big)\cdot e^t ~\leq ~ z'(t)~\leq~ -z(t)^2 +\Big(\vert \bar u(\bar x)\vert +2 \norm{\bar u}_{{\bf L}^1(\R)}\Big)\cdot e^t.
\]
To derive a lower bound on $t^*$ we look at the subsolution defined through 
\[
s'(t)~=~-s(t)^2-\Big(\vert \bar u(\bar x)\vert +2\norm{ \bar u}_{{\bf L}^1(\R)}\Big)\cdot e^t \quad,\quad s(0)~=~m.
\]
This implies that 
\[
\frac{1}{\sqrt{\vert \bar u(\bar x)\vert +2\norm{ \bar u}_{{\bf L}^1(\R)}}} \frac{\frac{s'(t)}{\sqrt{\vert \bar u(\bar x)\vert +2\norm{ \bar u}_{{\bf L}^1(\R)}}}}{\left(\frac{s(t)}{\sqrt{\vert \bar u(\bar x)\vert +2\norm{\bar u}_{{\bf L}^1(\R)}}}\right)^2+1}\geq -e^t \qquad \text{for all } t\in[0,t^*)
\]
or, equivalently, 
\[
\frac{1}{\sqrt{\vert \bar u(\bar x)\vert +2\norm{\bar u}_{{\bf L}^1(\R)}}}\left(\arctan\left(\frac{s(t)}{\sqrt{\vert \bar u(\bar x)\vert +2\norm{\bar u}_{{\bf L}^1(\R)}}}\right)\right)'\geq -e^t \]
for all $t\in [0,t^*)$. Thus a lower bound on $t^*$ is given by $T_l$, which is defined implicitly through
\[
\frac{1}{\sqrt{\vert \bar u(\bar x)\vert +2\norm{\bar u}_{{\bf L}^1(\R)}}} \left(-\frac{\pi}{2}-\arctan\left(\frac{m}{\sqrt{\vert \bar u(\bar x)\vert+2\norm{\bar u}_{{\bf L}^1(\R)}}}\right)\right)\geq 1-e^{T_l}.
\]
Hence
\[
T_l ~\geq~\ln\left(1+\frac{1}{\sqrt{|\bar u(\bar x)|+2\norm{\bar u}_{{\bf L}^1(\R)}}}\left( \frac{\pi}{2}+\arctan\left(\frac{\bar u'(\bar x)}{\sqrt{\vert \bar u(\bar x)\vert+2\norm{\bar u}_{{\bf L}^1(\R)}}}\right)\right)\right)\\
\]
which finally implies (\ref{lowb})\,.
\quad\\
\quad\\
\noindent
To derive an upper bound on $t^*$ we look at the supersolution defined through 
\[ 
s'(t)~=~-s(t)^2+\Big(\vert \bar u(\bar x)\vert +2\norm{ \bar u}_{{\bf L}^1(\R)}\Big)\cdot e^t \quad,\quad s(0)~=~m.
\]
Let $s_1(t)=s(t)e^{-t}$, then 
\begin{align*}
s_1'(t)&~=~s'(t)e^{-t}-s(t)e^{-t}\\
&~=~-s(t)^2e^{-t}-s(t)e^{-t}+\Big(\vert \bar u(\bar x)\vert +2\norm{ \bar u}_{{\bf L}^1(\R)}\Big)\\
&~=~ -s_1(t)^2e^t-s_1(t)+\Big(\vert \bar u(\bar x)\vert +2\norm{\bar u}_{{\bf L}^1(\R)}\Big)\\
&~\leq~ -s_1(t)^2-s_1(t)+\Big(\vert \bar u(\bar x)\vert +2\norm{\bar u}_{{\bf L}^1(\R)}\Big).
\end{align*}
Observe that $s_1'(t)$ is decreasing if 
\[ 
-s_1(t)^2-s_1(t)+\vert  \bar u(\bar x)\vert +2\norm{\bar u}_{{\bf L}^1(\R)}~=~ -\Big(s_1(t)+\frac12\Big)^2+\Big(\vert \bar u(\bar x)\vert +2\norm{\bar u}_{{\bf L}^1(\R)}+\frac14\Big)~<~ 0\,.
\]
Hence if we assume that $s(0)=s_1(0)< -\frac12 -\sqrt{\vert \bar u(\bar x)\vert +2\norm{\bar u}_{{\bf L}^1(\R)}+\frac14}$, the function $s(t)$ will be strictly decreasing on $[0, t^*)$
and
\[ 
\frac{1}{\sqrt{\vert \bar u(\bar x)\vert +2\norm{\bar u}_{{\bf L}^1(\R)}+\frac14}}\frac{\frac{s_1'(t)}{\sqrt{\vert \bar u(\bar x)\vert +2\norm{\bar u}_{{\bf L}^1(\R)}+\frac14}}}{\left(\frac{s_1(t)+\frac12}{\sqrt{\vert \bar u(\bar x)\vert +2\norm{\bar u}_{{\bf L}^1(\R)}+\frac14}}\right)^2-1}~\leq~ -1\,.
\]
Equivalently,
\[
\frac{1}{2\sqrt{\vert \bar u(\bar x)\vert +2\norm{\bar u}_{{\bf L}^1(\R)}+\frac14}} \left(\ln\left( \frac{s_1(t)+\frac12-\sqrt{\vert \bar u(\bar x)\vert +2\norm{\bar u}_{{\bf L}^1(\R)}+\frac14}}{s_1(t)+\frac12+\sqrt{\vert \bar u(\bar x)\vert +2\norm{\bar u}_{{\bf L}^1(\R)}+\frac14}}\right)\right)'\leq -1.
\]
Thus an upper bound on $t^*$ is given by $T_u$, which is defined through 
\begin{equation*}
T_u= \frac1{2\sqrt{\vert \bar u(\bar x)\vert +2\norm{\bar u}_{{\bf L}^1(\R)}+\frac14}} \ln\left(1-2\frac{\sqrt{\vert \bar u(\bar x)\vert +2\norm{\bar u}_{{\bf L}^1(\R)}+\frac14}}{m+\frac12+\sqrt{\vert \bar u(\bar x)\vert +2\norm{\bar u}_{{\bf L}^1(\R)}+\frac14}}\right)\,.
\end{equation*}
Recalling \eqref{condblowup}, we finally obtain 
\begin{equation*}
T_u~\leq~\frac{2}{\Big|2m+1+2\cdot \sqrt{\vert \bar u(\bar x)\vert +2\norm{\bar u}_{{\bf L}^1(\R)}+\frac14} \Big|}~=:~T^*\,.
\end{equation*}
\end{proof}

As an immediate consequence we obtain the following corollary.

\begin{corollary}
Let $u(t,x)$ be the weak entropy solution of \eqref{BP} with initial data $u(0,\cdot)=\bar u(\cdot)\in {\bf C}^1(\R)\cap {\bf L}^1(\R)$. In addition, let 
\[ 
m~=~\inf_{x\in\R} \bar u'(x)\qquad \text{and}\qquad  M~=~\norm{\bar u}_{{\bf L}^1(\R)}.
\]
Then the following statements hold.
\begin{itemize}
\item[(i)]
$u_x(t,x)$ remains bounded for all $t\in [0,T^-[$ where 
\[
T^-~=~\ln\left(1+\frac{1}{\sqrt{|m|+3M}}\cdot \left(\frac{\pi}{2}-\arctan\left(\frac{\vert m\vert}{\sqrt{2M}}\right)\right)\right)\,.
\]
\item[(ii)]
If 
\bel{Cond2}
m~<~-1-\sqrt{\frac{5}{2}M+1}~\leq~0\,,
\eeq
then $u_x(t,x)$ becomes unbounded within the time interval $[0, T^+[$ where 
\begin{equation}\label{Tupperbound}
 T^+~=~\frac{2}{\Big|2m+1+2\sqrt{|m|+\frac{5}{2}M+\frac{1}{4}}\Big|}\,.
\end{equation}
\end{itemize}
\end{corollary}

\begin{proof}
Given $\bar x\in\R$, denote by $x(t)$ the characteristic through $\bar x$, i.e. $x'(t)=u(t,x(t))$ and $x(0)=\bar x$. 
Theorem~\ref{WB} (i) implies that $u_x(t,x(t))$ is bounded for all $t\in[0, t_{\bar x}[$ with 
\begin{equation*}
t_{\bar x}~=~\ln\left(1+\frac{1}{\sqrt{|\bar u(\bar x)|+2\norm{\bar u}_{{\bf L}^1(\R)}}}\left( \frac{\pi}{2}+\arctan\left(\frac{\bar u'(\bar x)}{\sqrt{\vert \bar u(\bar x)\vert+2\norm{\bar u}_{{\bf L}^1(\R)}}}\right)\right)\right)\,.
\end{equation*}
In particular, we have 
\begin{equation*}
t_{\bar x}~\geq~ \ln\left(1+\frac{1}{\sqrt{\vert \bar u(\bar x)\vert +2M}}\left(\frac{\pi}{2}-\arctan\left(\frac{\vert m \vert}{\sqrt{\vert \bar u(\bar x)\vert +2M}}\right)\right)\right).
\end{equation*}
In addition, it is well-known that 
\begin{equation}\label{Linftybound}
\norm{\bar u}_{{\bf L}^\infty(\R)}~\leq~\sqrt{2\vert m\vert M}~\leq~\vert m\vert +\frac{M}{2},
\end{equation}
and hence
\begin{equation*}
t_{\bar x}\geq \ln\left(1+\frac{1}{\sqrt{\vert m\vert +3M}}\left(\frac{\pi}{2}-\arctan\left(\frac{\vert m\vert }{\sqrt{2M}}\right)\right)\right).
\end{equation*}
Since the right hand side is independent of $\bar x$, we have shown the first part.

As far as the second part is concerned, observe first of all that $m<0$, since by assumption $\bar u\in {\bf C}^1(\R)\cap {\bf L}^1(\R)$. Moreover, \eqref{Cond2} and \eqref{Linftybound} imply that for all $x\in \R$ 
\begin{equation*}
\frac{1}{2} +\sqrt{\vert \bar u(x)\vert +\norm{\bar u}_{{\bf L}^1(\R)}+\frac14 }~\leq~\frac{1}{2}+\sqrt{\vert m\vert +\frac52 M+\frac14}~<~\vert m\vert. 
\end{equation*}
Thus there exists $\bar{x}\in\R$ such that 
\begin{equation}\label{breakcond}
\bar u'(\bar{x})~<~-\frac{1}{2} -\sqrt{\vert \bar u(\bar{x})\vert +\norm{\bar u}_{{\bf L}^1(\R)}+\frac14}\quad\mathrm{and}\quad \bar u'(x)~\geq~m\,, 
\end{equation}
and hence according to the proof of Theorem~\ref{WB} wave breaking occurs.
Let $\bar x\in\R$ such that \eqref{breakcond} is satisfied, then an upper bound for the maximal time of existence is given by \eqref{upbound}
\begin{equation*}
T_{\bar x}~=~\frac{2}{\Big|2\bar u'(\bar{x})+1+2\cdot \sqrt{\vert \bar u(\bar x)\vert +2\norm{\bar u}_{{\bf L}^1(\R)}+\frac14} \Big|}.
\end{equation*}
Hence, by recalling (\ref{breakcond}), we obtain
\begin{equation*}
T_{\bar x}~\leq~\frac{2}{\Big|2m+1+2\sqrt{|m|+\frac{5}{2}M+\frac{1}{4}}\Big|}\,,
\end{equation*}
where the right hand side is independent of $\bar x$ and hence it yields \eqref{Tupperbound}.
\end{proof}

Finally we want to show that if the ${\bf L}^{\infty}$-norm of the derivative of the initial data $\bar u$ is small then the corresponding entropy solution of (\ref{BP}) will remain smooth for a long time.

\begin{theorem}
Let $u(t,x)$ be the weak entropy solution of \eqref{BP} with $u(0,\cdot)=\bar u(\cdot)\in \mathbf{C}^1(\R)\cap {\bf L}^1(\R)$ and let $m=\|\bar u'\|_{\bf{L}^{\infty}(R)}$\,.  Then $\norm{u_x(t,.)}_{{\bf L}^\infty(\R)}$ remains bounded for all $t\in[0,\ln (1+\frac{1}{m})[$.
\end{theorem}

\begin{proof} Given any $\bar{x}\in \R$, denote by $x(t)$ the characteristic through $\bar{x}$ at time $t=0$, i.e., $x'(t)=u(t,x(t))$ and $x(0)=\bar x$. Then the function $z(t)=u_x(t,x(t))$ satisfies the differential equation
\[
z'(t)~=~-z(t)^2+[G_x*u_x(t,\cdot)](x(t))\,.
\]
This implies that 
\[
{d\over dt}|z(t)|~\leq~z^2(t)+\|u_x(t,\cdot)\|_{{\bf L}^{\infty}(\R)}~\leq~\|u_x(t,\cdot)\|_{{\bf L}^{\infty}(\R)}^2+\|u_x(t,\cdot)\|_{{\bf L}^{\infty}(\R)}\,.
\]
Let $Q(t)~=~\|u_x(t,\cdot)\|_{{\bf L}^{\infty}(\R)}$, then we have 
\[
Q'(t)~\leq~Q^2(t)+Q(t)\,
\]
or equivalently,
\[
\frac{d}{dt}~\ln\Big(\frac{Q(t)}{Q(t)+1}\Big)~\leq~1\,.
\] 
Thus,
\[
\ln\Big(\frac{Q(t)}{Q(t)+1}\Big)+~\ln\Big(\frac{Q(0)+1}{Q(0)}\Big)~\leq~t\,,
\]
which implies that 
\[
\ln\Big(\frac{Q(t)}{Q(t)+1}\Big)~\leq~t-\ln\Big(1+\frac{1}{m}\Big)\,.
\]
Thus if $Q(t)$ becomes unbounded at time $t^*$, the left hand side tends to zero as $t\uparrow t^*$ and in particular,
$0\leq t^*-\ln\Big(1+\frac{1}{m}\Big)$.  Or, in other words, 
$Q(t)$ remains bounded for all $t\in \Big[0,\ln\Big(1+\frac{1}{m}\Big)\Big[$.
\end{proof}

\end{document}